\documentclass[12pt]{amsart}
\usepackage[top=1in, bottom=1in, left=1in, right=1in]{geometry}
\usepackage{times}

\title{A note on the natural density of product sets}

\usepackage{amsmath}
\usepackage{amssymb}
\usepackage{amsthm}
\usepackage{color}
\usepackage{hyperref}
\usepackage{enumerate}
\hypersetup{colorlinks=true}

\newtheorem{thm}{Theorem}

\newtheorem*{cor*}{Corollary}

\newtheorem{lem}{Lemma}[section]

\theoremstyle{remark}
\newtheorem*{rmk}{Remark}

\newtheorem{que}{Question}

\newtheorem*{que*}{Question}

\newcommand{\eq}[1]{ \begin{equation}
	\begin{split}
	#1 
	\end{equation} }

\newcommand{\als}[1]{\begin{align*} #1 \end{align*} }

\newcommand{\N}{\mathbb{N}}
\newcommand{\Z}{\mathbb{Z}}

\newcommand{\dd}{\mathbf{d}}

\newcommand{\CA}{\mathcal{A}}
\newcommand{\CB}{\mathcal{B}}
\newcommand{\CC}{\mathcal{C}}

\newcommand{\CN}{\mathcal{N}}

\newcommand{\CP}{\mathcal{P}}

\newcommand{\eps}{\varepsilon}
\renewcommand{\phi}{\varphi}

\newcommand{\bs}\boldsymbol{}

\makeatletter
\def\moverlay{\mathpalette\mov@rlay}
\def\mov@rlay#1#2{\leavevmode\vtop{%
		\baselineskip\z@skip \lineskiplimit-\maxdimen
		\ialign{\hfil$\m@th#1##$\hfil\cr#2\crcr}}}
\newcommand{\charfusion}[3][\mathord]{
	#1{\ifx#1\mathop\vphantom{#2}\fi
		\mathpalette\mov@rlay{#2\cr#3}
	}
	\ifx#1\mathop\expandafter\displaylimits\fi}
\makeatother

\definecolor{myblue}{rgb}{0.09,0.32,0.44} 
\hypersetup{pdfborder={0 0 0},pdfborderstyle={},colorlinks=true,linkcolor=myblue,citecolor=myblue,urlcolor=blue}

\author{Sandro Bettin}
\address{Dipartimento di Matematica, Universit\`{a} di Genova, Via Dodecaneso 35, 16146 Genova, Italy}
\email{bettin@dima.unige.it}

\author{Dimitris Koukoulopoulos}
\address{D\'epartement de math\'ematiques et de statistique, Universit\'e de Montr\'eal, CP 6128 succ. Centre-Ville, Montr\'eal, QC H3C 3J7, Canada}
\email{koukoulo@dms.umontreal.ca}

\author{Carlo Sanna}
\address{Dipartimento di Scienze Matematiche, Politecnico di Torino, Corso Duca degli Abruzzi 24, 10129 Torino, Italy}
\email{carlo.sanna.dev@gmail.com}

\begin{document}

\maketitle

\begin{abstract}
Given two sets of natural numbers $\mathcal{A}$ and $\mathcal{B}$ of natural density $1$ we prove that their product set $\mathcal{A}\cdot \mathcal{B}:=\{ab:a\in\CA,\,b\in\CB\}$ also has natural density $1$. On the other hand, for any $\varepsilon>0$, we show there are sets $\CA$ of density $>1-\varepsilon$ for which the product set $\mathcal{A}\cdot\mathcal{A}$ has density $<\eps$. 
This answers two questions of Hegyv\'{a}ri, Hennecart and Pach.
\end{abstract}

\section{Introduction}
Given two sets of natural numbers $\mathcal{A}$ and $\mathcal{B}$, let $\mathcal{A}\cdot\mathcal{B} := \{ab : a \in \mathcal{A},\, b \in \mathcal{B}\}$ be their \emph{product set}. Also, for any positive integer $k$, let $\mathcal{A}^k$ denote the $k$-fold product $\CA\cdots\CA$.

The problem of studying the cardinality of product sets has long been of interest in mathematics. The classic {\it multiplication table problem} due to Erd\H{o}s~\cite{MR73619, MR0126424} asks for bounds on the cardinality $M_n$ of the $n\times n$ multiplication table, i.e., of the set $\{1,\dots,n\}^2$. Erd\H os showed that $M_n=o(n^2)$ and Ford~\cite{MR2434882}, following earlier  of Tenenbaum~\cite{MR928635}, determined the exact order of magnitude of $M_n$. More recently~\cite{MR3187928}, the second author of the present paper provided uniform bounds for $\#(\{1, \dots, n_1\}\cdots\{1, \dots, n_s\})$ holding for a wide range of $n_1, \dots, n_s\in \mathbb N$.

For more general sets $\CA$, the problem of estimating $\#(\mathcal{A} \cap [1,x])^2$ was studied by Cilleruelo, Ramana, and Ramar\'e~\cite{MR3640773}. For example, they studied this problem when $\mathcal{A}$ is the set of shifted primes, the set of sums of two squares, and the set of shifted sums of two squares. Moreover, they computed the (almost sure) asymptotic behavior for $\#\mathcal{A}^2$ when $\mathcal{A}$ is a random subset of $\{1, \dots, n\}$ that contains each element of $\{1, \dots , n\}$ independently with probability $\delta \in (0,1)$. Sanna~\cite{SannaAMH} extended this last result to the product of arbitrarily many sets.

Hegyv\'{a}ri, Hennecart and Pach~\cite{MR3939472}  considered the analogous problem for infinite sets of natural numbers. In this context, the role of the cardinality is played by the \emph{natural density} $\mathbf{d}(\mathcal{A})$ of  a set $\mathcal{A}$, defined as usual by 
\[
\dd(A) = \lim_{x\to\infty} \frac{\#\CA\cap[1,x]}{x} .
\]
They asked the following questions (Questions 3 and 2 of  \cite{MR3939472}, respectively):

\begin{que}\label{q1}
If $\mathcal{A}$ is a set of natural numbers of density $1$, is it true that $\CA^2$ also has density $1$?
\end{que}

\begin{que}\label{q2} 
Is it true that $\inf_{\CA\subset\N:\ \dd(\CA)=\alpha} \dd(\CA^2)=0$ for any $\alpha\in [0,1)$, or at least for $\alpha\in[0,\alpha_0)$ for some $\alpha_0\in(0,1)$?
\end{que}

Clearly, Question \ref{q1} has an affirmative answer if $1\in\mathcal A$, and Hegyv\'{a}ri, Hennecart and Pach showed that it also suffices that $\mathcal A$ contains an infinite subset of mutually coprime integers $a_1 <a_2<\cdots$ such that $\sum_{i=1}^\infty a_i^{-1}=+\infty$. Here, we show that the answer is ``yes'' in full generality.

\begin{thm}\label{thm:main}
Let $\mathcal{A}, \mathcal{B} \subseteq \mathbb{N}$.
If $\mathbf{d}(\mathcal{A}) = \mathbf{d}(\mathcal{B}) = 1$, then $\mathbf{d}(\mathcal{A}\cdot\mathcal{B}) = 1$.
\end{thm}

\begin{cor*}
If $\CA\subset\N$ is such that $\dd(\CA)=1$, then $\dd(\CA^k)=1$ for each $k=2,3,\dots$
\end{cor*}

\begin{rmk}
In fact, the case $\CA=\CB$ of Theorem \ref{thm:main} implies easily the general case. Indeed, if $\dd(\CA)=\dd(\CB)=1$, then $\dd(\CA\cap \CB)=1$. In addition, if $(\CA\cap \CB)^2$ has density $1$, then so does $\CA\cdot \CB$.
\end{rmk} 

As it will be clear from the proof, the difference in the density of $\dd(\CA^2)$ with respect to Erd\H os's multiplication table problem lies in the fact that many elements of $\CA^2$ come from very ``unbalanced'' products, meaning products $ab$ such that the sizes of $a$ and $b$ are completely different.

\medskip

Let us now turn to Question \ref{q2}. We will answer it in a strong form that shows, among other things, that the condition that $\dd(\CA)=1$ in Theorem~\ref{thm:main} cannot be relaxed.

\begin{thm}\label{thm:main2}
For $\alpha\in[0,1]$, we have 
\begin{equation*}
\inf_{\mathcal{A} \,\subseteq\, \mathbb{N} \colon \mathbf{d}(\mathcal{A}) \,=\, \alpha} \mathbf{d}(\mathcal{A}^2)=\begin{cases}
0 & \text{if $ \alpha<1$,}\\
1 & \text{if $ \alpha=1$.}
\end{cases}
\end{equation*}
\end{thm}

\subsection*{Acknowledgements} The third author wishes to thank Paolo Leonetti for bringing the article of Hegyv\'{a}ri-Hennecart-Pach~\cite{MR3939472} to his attention. In addition, the first two authors of the paper would like to thank 
the Mathematisches Forschungsinstitut Oberwolfach for the hospitality; the paper was partially written there while attending a workshop in November 2019.

S.B. is member of the INdAM group GNAMPA and his work is partially supported by PRIN 2017 ``Geometric, algebraic and analytic methods in arithmetic'' and by INdAM.

D.K. is partially supported by the Natural Sciences and Engineering Research Council of Canada (Discovery Grant  2018-05699) and by the Fonds de recherche du Qu\'ebec - Nature et technologies (projet de recherche en \'equipe - 256442).

C.S. was supported by an INdAM postdoctoral fellowship and is a member of the INdAM group GNSAGA, and of CrypTO, the Cryptography and Number Theory group of Politecnico di Torino.

\section{Preliminaries}

\subsection*{Notation}
Given an integer $n$, we write $P^-(n)$ and $P^+(n)$ for its smallest and largest prime factors, respectively, with the convention that $P^-(1)=\infty$ and $P^+(1)=1$. If $P^+(n)\le y$, we say that $n$ is {\it $y$-smooth}, and if $P^-(n)>y$, we say that it is {\it $y$-rough}.
As usual, we let $\Phi(x, y)$ denote the number of $y$-rough numbers in $[1,x]$.
Given any integer $n$, we may write it uniquely as $n=ab$ with $P^+(a)\le y<P^-(b)$. We then call $a$ and $b$ the  {\it $y$-smooth and $y$-rough part} of $n$, respectively. Finally, we let $\Omega(n)$ denote the number of prime factors of $n$ counted with multiplicity.

\medskip

We need some standard lemmas. We give their proofs for the sake of completeness.

\begin{lem}\label{lem:rough}
	For $x \geq y > 1$, we have $\Phi(x, y) \ll x/\log y$.
\end{lem}

\begin{proof} This follows for example from Theorem 14.2 in \cite{dk-book} with $f(n)=1_{P^-(n)>y}$. 
\end{proof}

\begin{lem}\label{lem:deltay}
Uniformly for $x\ge y^2\ge1$ and $u\ge1$, we have
\[
\#\{n\le x: \exists d|n\ \text{such that}\ P^+(d)\le y^{1/u}\ \text{and}\ d>y\} \ll x \cdot (e^{-u} +y^{-1/3}) .
\]	
\end{lem}

\begin{proof} Without loss of generality, $u\ge4$. Let $\CB$ denote the set of $n\in\Z\cap[1,x]$ that have a $y^{1/u}$-smooth divisor $d>y$. Given $n\in\CB$, let $p_1\le p_2\le\cdots \le p_k$ be the sequence of prime factors of $n$ of size $\le y^{1/u}$ listed in increasing order and according to their multiplicity. By our assumption on $n$, we must have $p_1\cdots p_k>y$. Let $j$ be the smallest integer such that $p_1\cdots p_j>y$. We must have $j\ge5$ because all factors $p_i$ are $\le y^{1/u}\le y^{1/4}$. We then set $a=p_1\cdots p_{j-2}$, $p=p_{j-1}$, and $b=n/(ap)$, so that $a>y/(p_{j-1}p_j)\ge \sqrt{y}$, $ap\le y$, and $P^+(a)\le p\le P^-(b)$. Consequently,
\begin{equation}\label{shiu1}
\#\CB\le \sum_{p\le y^{1/u}} \sum_{\substack{P^+(a) \le  p \\ \sqrt{y}<a\le y/p}} \sum_{\substack{b\le x/(ap) \\ P^-(b)\ge p}} 1 
	\ll \sum_{p\le y^{1/u}} \sum_{\substack{P^+(a) \le  p \\ a>\sqrt{y}}}   \frac{x}{ap\log p} 
\end{equation}
by Lemma \ref{lem:rough}. If we let $\eps_p=\min\{2/3,2/\log p\}$, then Rankin's trick implies that
\[
\frac{\#\CB}{x}  \ll \sum_{p\le y^{1/u}} \sum_{\substack{P^+(a) \le  p \\ a>\sqrt{y}}}   \frac{(a/\sqrt{y})^{\eps_p}}{ap\log p}  
		= \sum_{p\le y^{1/u}} \frac{y^{-\eps_p/2}}{p\log p} \sum_{P^+(a) \le  p} \frac{1}{a^{1-\eps_p}}.
\]
The sum over $a$ equals $\prod_{q\le p}(1-q^{-1+\eps_p})^{-1}$ with $q$ denoting a prime number. Since $q^{\eps_p}=1+O(\log q/\log p)$ for $q\le p$, Mertens' estimates \cite[Theorem 3.4]{dk-book} imply that the sum over $a$ is $\ll\log p$. We conclude that
\als{
\frac{\#\CB}{x} \ll y^{-1/3}+\sum_{100<p\le y^{1/u}} \frac{e^{-\log y/\log p} }{p} 
	&\le y^{-1/3}+\sum_{j\ge1} \sum_{y^{1/(u(j+1))}<p\le y^{1/(uj)}} \frac{e^{-ju}}{p}  \\
	&\ll y^{-1/3}+ \sum_{j\ge1} e^{-ju} \ll y^{-1/3}+e^{-u} 
}
using Mertens' estimates once again. This completes the proof.
\end{proof}

\begin{lem}\label{thm:Turan} Let $y\ge2$ and $\lambda\in[0,1.99]$, and set $Q(\lambda)=\lambda\log\lambda-\lambda+1$ for $\lambda > 0$ and $Q(0) = 0$. 
If $0\le \lambda\le1$, then 
\[
\prod_{p\le y}\bigg(1-\frac{1}{p}\bigg) \sum_{\substack{P^+(m)\le y \\ \Omega(m)\le \lambda\log\log y}} \frac{1}{m} \ll (\log y)^{-Q(\lambda)} ,
\]
whereas if $1\le\lambda\le 1.99$, then 
\[
\prod_{p\le y}\bigg(1-\frac{1}{p}\bigg) \sum_{\substack{P^+(m)\le y \\ \Omega(m)\ge \lambda\log\log y}}\frac{1}{m} \ll (\log y)^{-Q(\lambda)} .
\]
\end{lem}

\begin{proof} The result is trivial if $\lambda=0$ by Mertens' estimates \cite[Theorem 3.4]{dk-book}, so assume that $\lambda>0$. If $0<\lambda\le1$, then
\als{
\sum_{\substack{P^+(m)\le y \\ \Omega(m)\le \lambda\log\log y}} \frac{1}{m} 
	\le \sum_{P^+(m)\le y}\frac{\lambda^{\Omega(m)-\lambda\log\log y}}{m} 
	&= (\log y)^{-\lambda\log\lambda} \prod_{p\le y} \bigg(1-\frac{\lambda}{p}\bigg)^{-1}  \\
	&\asymp (\log y)^{-Q(\lambda)}\prod_{p\le y} \bigg(1-\frac{1}{p}\bigg)^{-1} 
}
where we used Mertens' estimates once again. Similarly, if $1\le\lambda\le1.99$, then
\[
\sum_{\substack{P^+(m)\le y \\ \Omega(m)\ge \lambda\log\log y}} \frac{1}{m} 
\le \sum_{P^+(m)\le y}\frac{\lambda^{\Omega(m)-\lambda\log\log y}}{m} 
\asymp (\log y)^{-Q(\lambda)}\prod_{p\le y} \bigg(1-\frac{1}{p}\bigg)^{-1} . 
\]
This completes the proof.
\end{proof}

\begin{lem}\label{lem:sieve}
Let $\CP$ be a set of primes such that $\sum_{p\in\CP}1/p<\infty$. Then
\[
\dd\big(\{n\in\N:p|n\ \Rightarrow p\notin\CP\}\big) = \prod_{p\in\CP} \bigg(1-\frac{1}{p}\bigg) .
\]
\end{lem}

\begin{proof}
The number of integers $n\le x$ with a prime divisor $p>\log x$ from $\CP$ is 
\[
\le \sum_{p>\log x,\, p\in\CP} \frac{x}{p} =o(x)\qquad\text{as}\quad x\to\infty,
\]
because $\sum_{p\in\CP}1/p$ converges. Hence, if we write $\CP'=\CP\cap[1,\log x]$, then 
\[
\#\{n\le x: p|n\ \Rightarrow p\notin\CP\}=\#\{n\le x:p|n\ \Rightarrow p\notin\CP'\}+o(x) = x\prod_{p\in\CP'}\bigg(1-\frac{1}{p}\bigg) + o(x)  
\]
from the inclusion-exclusion principle that has $\le 2^{\#\CP'}\le 2^{\log x}=o(x)$ steps (e.g., see \cite[Theorem 2.1]{dk-book}). 
Since $\prod_{p\in\CP\setminus\CP'}(1-1/p)\sim1$ by our assumption that $\sum_{p\in\CP}1/p<\infty$, the proof is complete.
\end{proof}

\section{Proof of Theorem~\ref{thm:main}}

Assume $x$ is sufficiently large and let $y = y(x)$ and $u = u(x)$ to be chosen later, with $y,u \to +\infty$ slowly as $x \to +\infty$. In particular, $y\le \sqrt{x}$. In the following, for the sake of notation, we will often omit the dependence on $x,y,u$.

With a small abuse of notation, given an integer $n$, let $n_{\mathrm{smooth}}$ denote its $y^{1/u}$-smooth part and let $n_{\mathrm{rough}}$ denote its $y^{1/u}$-rough part. We then set
\[
\mathcal{N}=\{ n \leq x: n_{\mathrm{smooth}} \le  y \} .
\]
By Lemma~\ref{lem:deltay}, we have $\#\CN\sim x$ as $x\to\infty$. Therefore, in order to prove Theorem~\ref{thm:main}, it is enough to show that
\[
\#\mathcal{C} = o(x), 
\quad\text{where}\quad 
	\mathcal{C} := \CN\setminus (\mathcal{A}\cdot\mathcal{B}).
\]

Let $n\in\CC$. Since $n = n_{\mathrm{smooth}} \cdot n_{\mathrm{rough}}$, we must have that  either $n_{\mathrm{smooth}}\notin \CA$ or $n_{\mathrm{rough}}\notin \CB$. Consequently,
\[
\#\CC\le S_1+S_2
\]
with
\[
S_1 := \#\{ n  \in\CN:  n_{\mathrm{smooth}}\notin \CA\} 
	\quad\text{and}\quad 
S_2 := \#\{ n  \in\CN:  n_{\mathrm{rough}}\notin \CB\} .
\]

Let us first bound $S_1$. Letting $m=n_{\mathrm{smooth}}$, we have
\begin{equation*}
S_1 \leq \sum_{m\le y,\, m\notin \CA} \Phi(x/m,y^{1/u})  
	\ll \frac{ux}{\log y} \sum_{m\le y ,\, m\notin \CA}\frac{1}{m}
\end{equation*}
by Lemma \ref{lem:rough}. Since we have assumed that $\dd(\CA)=1$, we must have that $\dd(\N\setminus\CA)=0$ and thus
\begin{equation*}
\alpha(t) := \frac1{\log t} \sum_{m\le t,\, m\notin \CA} \frac{1}{m} \ \rightarrow\  0 \qquad\text{as}\quad t\to\infty .
\end{equation*}
Hence, setting $u=u(y) := \alpha(y)^{-1/2}$, we have $u \to +\infty$ and $S_1=o(x)$ as $x \to +\infty$.

Let us now bound $S_2$. Writing $m'=n_{\mathrm{rough}}$, we have 
\begin{equation*}
S_2 \leq \sum_{m\le y} \#\{m'\le x/m: m'\notin\CB\} .
\end{equation*}
By hypothesis, we have $\dd(\CB)=1$, so that $\dd(\N\setminus\CB)=0$. Thus 
\begin{equation*}
\beta(t) := \sup_{s \,\geq\, t} \frac{\#\big((\N\setminus \CB)\cap[1,s]\big)}{s} \ \rightarrow\  0 \qquad\text{as}\quad t\to\infty .
\end{equation*}
Hence, setting $y := \min\big( x^{1/2}, \exp\big(\beta(x^{1/2})^{-1/2}\big)\big)$, we have $y \to +\infty$ as $x \to +\infty$ and
\begin{equation*}
S_2 \leq \sum_{d \,\leq\, y} \beta(x/d) \cdot \frac{x}{d} \leq x \beta(x/y) \sum_{d \,\leq\, y} \frac1{d} \ll x\beta(x^{1/2}) \log y \leq x \beta(x^{1/2})^{1/2} = o(x) .
\end{equation*}
In conclusion, $\#\mathcal{C} = o(x)$, as desired.\qed

\begin{rmk}
The proof of Theorem~\ref{thm:main} can be made quantitative. For example, if one has  $\#\{n\le x: n\notin\CA\}, \#\{n\le x: n\notin \CB\} \ll x (\log x)^{-a}$ for some fixed $0<a<1$, then taking $y=\exp\big((\log x)^{\frac{a}{1+a}}\big)$ and $u=\log\log x$ in the above argument yields
\begin{align*}
\#\{n\le x:n\notin\CA\cdot\CB\}  \ll xe^{-u} +\frac{x u}{(\log y)^a}+\frac{x\log y}{(\log x)^a}\ll x(\log x)^{-\frac{a^2}{1+a}+o(1)} .
\end{align*}
An interesting question is to determine the optimal exponent of $\log x$ in this upper bound.
\end{rmk}

\section{Proof of Theorem~\ref{thm:main2}}
The case $\alpha=1$ follows from Theorem~\ref{thm:main}, whereas for the case $\alpha=0$ one can just observe that $\mathbf{d}(\emptyset)=\mathbf{d}(\emptyset^2)=0$. We may thus assume $\alpha\in(0,1)$. Given any $\varepsilon>0$, we need to construct a set $\mathcal A$ of density $\alpha$ such that the density of $\mathcal A^2$ exists and is smaller than $\varepsilon$.

Let $k\in\mathbb N$, $y\ge1$ and a set of primes $\CP\subset(y,+\infty)$ with $\sum_{p \in \CP} 1 / p < \infty$ to be chosen later. Using the notation $\Omega_y(n) =\sum_{p^a|n,\, p\le y} 1$, let us consider the sets
\begin{equation*}
\CB_{y,k,\CP} := \big\{n \in \N  : \Omega_y(n) \geq  k,\ (n,p)=1\ \forall p\in\CP \big\} .
\end{equation*}
The key property these sets have is that $\CB_{y,k,\CP}^2=\CB_{y,2k,\CP}$.

Now, using Lemma \ref{lem:sieve} twice (once, with $\CP_{\mathrm{Lemma}\ \ref{lem:sieve}}=\CP\cup\{p\le y\}$ and once with $\CP_{\mathrm{Lemma}\ \ref{lem:sieve}}=\{p\le y\}$), we find that
\begin{align*}
\dd(\CB_{y,k,\CP})
&=\prod_{p\in\CP}\bigg(1-\frac1p\bigg)\prod_{p\leq y}\bigg(1-\frac1p\bigg)\sum_{\substack{P^+(m)\le y  \\ \Omega(m)\geq k}} \frac1m=  \dd(\CB_{y,k,\emptyset})  \prod_{p\in\CP}\bigg(1-\frac{1}{p}\bigg).
\end{align*}
Similarly,
\begin{align*}
\mathbf{d}(\mathcal{B}_{y,k,\mathcal P}^2)=\mathbf{d}(\mathcal{B}_{y,2k,\mathcal P})&=\prod_{p\in\mathcal P}\bigg(1-\frac1p\bigg)\mathbf{d}(\mathcal{B}_{y,2k,\emptyset}).
\end{align*}

Now, take $y:=\exp(\exp(4k/3))$, so that $k=\tfrac34\log\log y$. For any fixed $\varepsilon>0$, Lemma~\ref{thm:Turan} implies that if $k$ is sufficiently large in terms of $\alpha$ and $\eps$, then $\mathbf{d}(\mathcal{B}_{y,k,\emptyset})>\alpha$ and $\dd(\CB_{y,2k,\emptyset})<\varepsilon$. Let us fix for the remainder of the proof such a choice of $k$. We then construct $\CP$ in the following way: we take $p_1>y$ to be the smallest prime such that $(1-1/p_1)\dd(\CB_{y,k,\emptyset})>\alpha$, $p_2>p_1$ the smallest prime such that $(1-1/p_1)(1-1/p_2)\dd(\CB_{y,k,\emptyset})>\alpha$ and so on. Taking $\CP:=\{p_1,p_2,\dots\}$ we clearly have $\dd(\CB_{y,k,\emptyset})\prod_{p\in\CP}(1-1/p)=\alpha$. Thus, $\dd(\CB_{y,k,\CP})=\alpha$ and $\dd(\CB_{y,k,\CP}^2)<\varepsilon$, as desired.\qed

\begin{rmk}
If $\mathbf{d}(\mathcal{A}^2)$ in Theorem~\ref{thm:main2} is replaced by the upper density  $\overline{\mathbf{d}}(\mathcal{A}^2)$, then one could just take $\mathcal A$ to be any density $\alpha$ subset of $ \big\{n \in \mathbb{N} : \Omega_y(n) \geq  \frac34\log\log y\big\} $ for $y$ large enough. However, in general there is no guarantee that $\mathcal A^2$ has asymptotic density. For this reason, in order to prove Theorem~\ref{thm:main2}, it is more convenient to construct explicit suitable sets $\mathcal A$.
\end{rmk}

\bibliographystyle{plain}

\begin{thebibliography}{99}

\bibitem{MR3640773}
J.~Cilleruelo, D.~S. Ramana, and O.~Ramar{\'e}, \textit{Quotient and product sets
  of thin subsets of the positive integers}, Proc. Steklov Inst. Math.
  \textbf{296} (2017), no.~1, 52--64.

\bibitem{MR73619}
P.~Erd\H{o}s, \textit{Some remarks on number theory}, Riveon Lematematika
  \textbf{9} (1955), 45--48.

\bibitem{MR0126424}
\bysame,\, \textit{An asymptotic inequality in the theory of numbers}, Vestnik
  Leningrad. Univ. \textbf{15} (1960), no.~13, 41--49.

\bibitem{MR204378}
\bysame,\, \emph{On some properties of prime factors of integers}, Nagoya
  Math. J. \textbf{27} (1966), 617--623.

\bibitem{MR2434882}
K.~Ford, \textit{The distribution of integers with a divisor in a given
  interval}, Ann. of Math. (2) \textbf{168} (2008), no.~2, 367--433.

\bibitem{MR3939472}
N.~Hegyv\'{a}ri, F.~Hennecart, and P.~P. Pach, \emph{On the density of sumsets
  and product sets}, Australas. J. Combin. \textbf{74} (2019), 1--16.

\bibitem{MR3187928}
D.~Koukoulopoulos, \textit{On the number of integers in a generalized
  multiplication table}, J. Reine Angew. Math. \textbf{689} (2014), 33--99.

\bibitem{dk-book}\bysame,\,
{\it The distribution of prime numbers.} 
Graduate Studies in Mathematics, 203. American Mathematical Society, Providence, RI, 2019.

\bibitem{SannaAMH}
C.~Sanna, \emph{A note on product sets of random sets}, Acta Math. Hungar., to appear.

  \bibitem{MR928635}
G.~Tenenbaum, \emph{Un probl\`eme de probabilit\'e conditionnelle en arithm\'etique},
Acta Arith. 49 (1987), no. 2, 165--187.

\end{thebibliography}

\end{document}